\documentclass[12pt, reqno]{amsart}
\usepackage{amsmath, amstext, amsbsy, amssymb, amscd}
\usepackage{amsmath}
\usepackage{amsxtra}
\usepackage{amscd}
\usepackage{amsthm}
\usepackage{amsfonts}
\usepackage{amssymb}
\usepackage{color}
\usepackage{eucal}
\usepackage{epic}
\usepackage{tikz}
\usepackage{youngtab}

\newtheorem{theorem}{Theorem}[section]

\newtheorem{lemma}[theorem]{Lemma}
\newtheorem{proposition}[theorem]{Proposition}
\newtheorem{corollary}[theorem]{Corollary}
\theoremstyle{definition}
\newtheorem{definition}[theorem]{Definition}
\newtheorem{example}[theorem]{Example}

\newtheorem{remark}[theorem]{Remark}

\definecolor{A}{rgb}{.75,1,.75}

\numberwithin{equation}{section}

\newcommand{\ad}{\operatorname{ad}}

\newcommand{\del}{\delta}
\newcommand{\Del}{\Delta}
\newcommand{\ev}{\operatorname{ev}}
\newcommand{\W}{\mathcal{W}}
\newcommand{\End}{\operatorname{End}}

\newcommand{\mb}{\mathfrak{b}}
\newcommand{\ep}{\epsilon}
\newcommand{\gl}{\mathfrak{gl}}
\newcommand{\sla}{\mathfrak{sl}}

\newcommand{\g}{\mathfrak{g}}

\newcommand{\gr}{\operatorname{gr}}
\newcommand{\glMN}{\mathfrak{gl}_{M|N}}

\newcommand{\la}{\lambda}

\newcommand{\pa}{\overline}
\newcommand{\sgn}{\operatorname{sgn}}
\newcommand{\rdet}{\operatorname{rdet}}

\newcommand{\pr}{\operatorname{pr}}
\newcommand{\col}{\text{col}}
\newcommand{\ov}{\overline}
\newcommand{\row}{\text{row}}

\newdimen\Hoogte    \Hoogte=12pt
\newdimen\Breedte   \Breedte=12pt
\newdimen\Dikte     \Dikte=0.5pt

\newenvironment{Young}{\begingroup
       \def\vr{\vrule height0.8\Hoogte width\Dikte depth 0.4\Hoogte}
       \def\fbox##1{\vbox{\offinterlineskip
                    \hrule height\Dikte
                    \hbox to \Breedte{\vr\hfill##1\hfill\vr}
                    \hrule height\Dikte}}
       \vbox\bgroup \offinterlineskip \tabskip=-\Dikte \lineskip=-\Dikte
            \halign\bgroup &\fbox{##\unskip}\unskip  \crcr }
       {\egroup\egroup\endgroup}
\def\Diagram#1{\relax\ifmmode\vcenter{\,\begin{Young}#1\end{Young}\,}\else%
              $\vcenter{\,\begin{Young}#1\end{Young}\,}$\fi}

\begin{document}
\title[Finite $W$-superalgebras and truncated super Yangians] {Finite $W$-superalgebras and truncated super Yangians}

\author[Yung-Ning Peng]{Yung-Ning Peng}
\address{Institute of Mathematics, Academia Sinica,
Taipei City, Taiwan, 10617} \email{ynp@math.sinica.edu.tw}


\begin{abstract}
Let $Y_{m|n}^{\ell}$ be the super Yangian of general linear Lie superalgebra for $\gl_{m|n}$. Let $e\in\gl_{m\ell|n\ell}$ be a ``rectangular" nilpotent element and $\W_e$ be the finite $W$-superalgebra associated to $e$. We show that $Y_{m|n}^{\ell}$ is isomorphic to $\W_e$.
\end{abstract}


\maketitle


\section{Introduction}
The connection between Yangians and finite $W$-algebras of type A was first noticed by mathematical physicists Briot, Ragoucy and Sorba in \cite{RS, BR1} under some restrictions, and then constructed in general cases by Brundan and Kleshchev explicitly in \cite{BK}. In this way, a realization of finite $W$-algebra in terms of truncated Yangian is obtained, and this provides a useful tool for the study of the representation theory of finite $W$-algebras. In this note, we establish such a connection between finite $W$-superalgebras and super Yangians of type A where the nilpotent element $e$ is {\em rectangular} (cf. \textsection 3 for the precise definition).

Let $Y_{m|n}$ be the super Yangian for $\gl_{m|n}$ and $Y_{m|n}^{\ell}$ be the truncated super Yangian of level $\ell$ for some non-negative integer $\ell$. Our main result is that there exists an isomorphism of filtered superalgebras between $Y_{m|n}^{\ell}$ and $\W_e$, the finite $W$-superalgebra associated to a rectangular $e\in \gl_{m\ell|n\ell}$. Such a connection was firstly obtained in \cite{BR3}. In this article, we provide a new proof in a different approach, which is similar to \cite{BK}. 

In a recent paper \cite{BBG}, such a connection between the super Yangian $Y_{1|1}$ and the finite $W$-superalgebra associated to a {\em principal} nilpotent element $e$ is obtained. In particular, the specialization of our result when $m=n=1$ coincides with the special case in \cite{BBG} when $e$ is both principal and rectangular.

\section{Super Yangian $Y_{m|n}$ and its truncation $Y_{m|n}^{\ell}$}
Let $m$ and $n$ be non-negative integers. Let $\mb$ be an $\ep$-$\del$ sequence of $\gl_{m|n}$ introduced in \cite{FSS1,LS,CW}. To be precise, $\mb$ is a sequence consisting of exactly $m$ $\del$'s and $n$ $\ep$'s, both indistinguishable. 
For example, $\del\ep\del\ep\ep$ is such a sequence of $\gl_{2|3}$.

Note that the set of such sequences is in one-to-one correspondence with the classes of the simple systems of $\gl_{m|n}$ modulo the Weyl group action, and each $\ep$-$\del$ sequence $\mb$ gives rise to a distinguished Borel subalgebra, which will be also denoted by $\mb$. In particular, the sequence $\mb^{st}:=\stackrel{m}{\overbrace{\delta\ldots\delta}}\,\stackrel{n}{\overbrace{\epsilon\ldots\epsilon}}$ corresponds to the standard Borel subalgebra consisting of upper triangular matrices.

\begin{remark}
Contrary to the $\gl_n$ case, there are many inequivalent simple root systems and hence many inequivalent Dynkin diagrams for $\gl_{m|n}$. It is also true that each $\ep$-$\del$ sequence corresponds to exactly one of the Dynkin diagrams. As an example, the following are two inequivalent Dynkin diagrams for $\gl_{2|3}$ (equivalently, $\sla_{2|3}$), where their corresponding simple systems and $\ep$-$\del$ sequences are listed:

\begin{center}
\begin{equation*}
\hskip -3cm \setlength{\unitlength}{0.16in}
\begin{picture}(20,1)
\put(0,0){\makebox(0,0)[c]{$\bigcirc$}}
\put(3.5,0){\makebox(0,0)[c]{$\bigotimes$}}
\put(7,0){\makebox(0,0)[c]{$\bigcirc$}}
\put(10.5,0){\makebox(0,0)[c]{$\bigcirc$}}

\put(17,0){\makebox(0,0)[c]{$\bigotimes$}}
\put(20.5,0){\makebox(0,0)[c]{$\bigotimes$}}
\put(24,0){\makebox(0,0)[c]{$\bigotimes$}}
\put(27.5,0){\makebox(0,0)[c]{$\bigcirc$}}

\put(0.5,0){\line(1,0){2.5}}
\put(4,0){\line(1,0){2.5}}
\put(7.5,0){\line(1,0){2.5}}

\put(17.5,0){\line(1,0){2.5}}
\put(21,0){\line(1,0){2.5}} 
\put(24.5,0){\line(1,0){2.5}}

\put(0,-1.3){\makebox(0,0)[c]{\tiny$\del_1-\del_2$}}
\put(3.5,-1.3){\makebox(0,0)[c]{\tiny$\del_2-\ep_1$}}
\put(7,-1.3){\makebox(0,0)[c]{\tiny$\ep_1-\ep_2$}}
\put(10.5,-1.3){\makebox(0,0)[c]{\tiny$\ep_2-\ep_3$}}

\put(17,-1.3){\makebox(0,0)[c]{\tiny$\del_1-\ep_1$}}
\put(20.5,-1.3){\makebox(0,0)[c]{\tiny$\ep_1-\del_2$}}
\put(24,-1.3){\makebox(0,0)[c]{\tiny$\del_2-\ep_2$}}
\put(27.5,-1.3){\makebox(0,0)[c]{\tiny$\ep_2-\ep_3$}}
\put(5,1.5){\makebox(0,0)[c]{$\mb=\mb^{st}=\del\del\ep\ep\ep$}}
\put(22,1.5){\makebox(0,0)[c]{$\mb=\del\ep\del\ep\ep$}}

\end{picture}
\end{equation*}
\vskip 0.8cm
\end{center}
Here $\bigcirc$ denotes an even simple root, $\bigotimes$ denotes an odd simple root, $\del_i$ and $\ep_j$ are elements of $\mathfrak{h}^*$ sending a matrix to its $i$-th and $(2+j)$-th diagonal entry, respectively. Such a phenomenon can also be observed in other types of Lie superalgebras. We refer the reader to \cite{FSS2,CW} for the detail.
\end{remark}

For each $1\leq i\leq m+n$, define a number $|i|\in\mathbb{Z}_2$, called the parity of $i$, as follows:
\begin{equation}\label{parity}
|i|:= \left\{
\begin{array}{ll}
\pa{0} &\hbox{if the $i$-th position of $\mb$ is $\del$,}\\
\pa{1} &\hbox{if the $i$-th position of $\mb$ is $\ep$.}
\end{array}
\right.
\end{equation}

For a given $\mb$, the super Yangian $Y_{m|n}$ (cf. \cite{Na}) is the associative $\mathbb{Z}_2$-graded algebra (i.e., superalgebra) over $\mathbb{C}$ with generators
\begin{equation}\label{RTTgen}
\left\lbrace t_{ij}^{(r)}\,| \; 1\le i,j \le m+n; r\ge 1\right\rbrace,
\end{equation}
subject to certain relations.

To write down the relations, we firstly define the parity of $t_{ij}^{(r)}$ to be $|i|+|j|$, and $t_{ij}^{(r)}$ is called an even (odd, respectively) element if its parity is $\ov{0}$ ($\ov{1}$, respectively).
 
The defining relations are
\begin{equation}\label{RTT}
[t_{ij}^{(r)}, t_{hk}^{(s)}] = (-1)^{|i|\,|j| + |i|\,|h| + |j|\,|h|}
\sum_{t=0}^{\mathrm{min}(r,s) -1} \Big( t_{hj}^{(t)}\, t_{ik}^{(r+s-1-t)} -   t_{hj}^{(r+s-1-t)}\, t_{ik}^{(t)} \Big),
\end{equation}
where the bracket is understood as the supercommutator. 
By convention, we set $t_{ij}^{(0)}:=\del_{ij}$.

In the case when $m=0$ or $n=0$, it reduces to the usual Yangian. The generators in (\ref{RTTgen}) are called the RTT generators while the relations (\ref{RTT}) are called the RTT relations. As in the case of $\gl_{m|n}$, $Y_{m|n}$ are isomorphic for all $\mb$. Note that the original definition in \cite{Na} corresponds to the case when $\mb=\mb^{st}$. 

For all $1\leq i,j\leq m+n$, we define the formal power series 
\[
t_{ij}(u):= \sum_{r\geq 0} t_{ij}^{(r)}u^{-r}.
\]
It is well-known (cf. \cite{Na}) that $Y_{m|n}$ is a Hopf-algebra, where the comultiplication 
$\Del:Y_{m|n}\rightarrow Y_{m|n}\otimes Y_{m|n}$ is defined by 
\begin{equation}\label{Del}
\Del(t_{ij}^{(r)})=\sum_{s=0}^r \sum_{k=1}^{m+n} t_{ik}^{(r-s)}\otimes t_{kj}^{(s)},
\end{equation}
and one has the evaluation homomorphism $\ev:Y_{m|n}\rightarrow U(\gl_{m|n})$ defined by
\begin{equation}\label{ev}
\ev\big(t_{ij}(u)\big):= \del_{ij} + (-1)^{|i|} e_{i,j},
\end{equation}
where $e_{i,j}$ denotes the elementary matrix in $\gl_{m|n}$.

\begin{definition}
Let $\ell$ be a non-negative integer and $I_\ell$ be the 2-sided ideal of $Y_{m|n}$ generated by the elements $\{ t_{ij}^{(r)}|1\leq i,j\leq m+n, r>\ell \,\}$. The {\em truncated super Yangian of level $\ell$}, denoted by $Y_{m|n}^{\ell}$, is defined to be the quotient $Y_{m|n}/I_\ell$.
\end{definition}

Note that $Y_{m|n}$, $I_\ell$ and the quotient $Y_{m|n}^{\ell}$ are all independent of the choice of $\mb$. The image of $t_{ij}^{(r)}$ in the quotient $Y_{m|n}^{\ell}$ will be denoted by $t_{ij;\mb}^{(r)}$, since it will be identified with certain element depending on $\mb$ later; see (\ref{tdef}).

Define $\kappa_1=\ev$, and for any integer $\ell\geq 2$, we define the following homomorphism 
\begin{equation}\label{kappa}
\kappa_\ell:=(\overbrace{\ev\otimes\cdots\otimes\ev}^{\ell-copies})\circ \Del^{\ell-1}:Y_{m|n}\rightarrow U(\gl_{m|n})^{\otimes \ell},
\end{equation}
then we have
\begin{equation}\label{kpimg}
\kappa_\ell(t_{ij}^{(r)})=\sum_{1\leq s_1<\cdots<s_r\leq \ell}\,\,\sum_{1\leq i_1,\ldots, i_{r-1}\leq m+n} 
(-1)^{|i|+|i_1|+\cdots+|i_{r-1}|}
e_{i,i_1}^{[s_1]}e_{i_1,i_2}^{[s_2]}\cdots e_{i_{r-1},j}^{[s_r]},
\end{equation}
where $e_{i,j}^{[s]}:= 1^{\otimes(s-1)}\otimes e_{i,j}\otimes 1^{\otimes (\ell-s)}$.

The following proposition is a PBW theorem for $Y_{m|n}^{\ell}$ and $Y_{m|n}$. 
\begin{proposition}\cite[Theorem 1]{Go}\label{PBWSY}
Let $\kappa_\ell$, $I_\ell$ be as above.
\begin{enumerate}
\item[(1)] The kernel of $\kappa_\ell$ is exactly $I_{\ell}$.
\item[(2)]
The set of all supermonomials in the elements of $Y_{m|n}^{\ell}$
\[
\left\lbrace t_{ij;\mb}^{(r)}| 1\leq i,j\leq m+n, 1\leq r\leq \ell\right\rbrace
\] taken in some fixed order
forms a basis for $Y_{m|n}^{\ell}$.
\item[(3)]
The set of all supermonomials in the elements of $Y_{m|n}$
\[
\left\lbrace t_{ij}^{(r)}| 1\leq i,j\leq m+n,  r\geq 1 \right\rbrace
\] taken in some fixed order
forms a basis for $Y_{m|n}$.
\end{enumerate}
\end{proposition}

 As a consequence, we may identify $Y_{m|n}^{\ell}$ with the image of $Y_{m|n}$ under the map $\kappa_\ell$. In particular, the induced homomorphism
\begin{equation}\label{kappainj}
\kappa_\ell:Y_{m|n}^{\ell}\rightarrow U(\gl_{m|n})^{\otimes \ell}
\end{equation}
is injective.

Define the canonical filtration of $Y_{m|n}$ 
\[
F_0Y_{m|n}\subseteq F_1Y_{m|n}\subset F_2Y_{m|n}\subseteq \cdots
\]
by deg $t_{ij}^{(r)}:=r$, i.e., $F_dY_{m|n}$ is the span of all supermonomials in the generators $t_{ij}^{(r)}$ of total degree $\leq$ d. It is clear from (\ref{RTT}) that the associated graded algebra $\gr Y_{m|n}$ is supercommutative. We also obtain the canonical filtration on $Y_{m|n}^{\ell}$ induced from the natural quotient map $Y_{m|n}\twoheadrightarrow Y_{m|n}^{\ell}$. 

\section{Finite $W$-superalgebras}
Let $M$ and $N$ be non-negative integers and let $\mathfrak{g}=\glMN$. Every elements of $\g$ in the context are considered $\mathbb{Z}_2$-homogeneous unless mentioned specifically. Recall that $\mathfrak{g}$ acts on $\mathbb{C}^{M|N}$ via the standard representation, which we will denote by $\psi:\mathfrak{g}\rightarrow \End(\mathbb{C}^{M|N})$. Let $(\, \cdot \, ,\,\cdot\,)$ denote the non-degenerate even supersymmetric invariant bilinear form on $\g$ defined by 
\begin{equation}\label{strform}
(x,y):=str\big(\psi(x)\circ\psi(y)\big),\,\, \forall x,y \in \g,
\end{equation}
where $str$ means the super trace. See \cite{CW,FSS2,Ka} for more comprehensive studies about Lie superalgebras.

\begin{remark}
There exists another well-known even supersymmetric invariant bilinear form on $\mathfrak{g}$, called the {\em Killing form}, defined by 
\begin{equation*}
<x,y>:= str\big( \Psi(x)\circ \Psi(y) \big), \,\,\forall x,y\in\mathfrak{g},
\end{equation*}
where $\Psi:\mathfrak{g}\rightarrow \End \mathfrak{g}$ denotes the adjoint representation of $\mathfrak{g}$. The Killing form is non-degenerate for all $M\neq N$ and it plays an important rule in the classification of simple Lie superalgebras \cite{Ka}; however, it is degenerate if $M=N$ and so it is not suitable for our purpose. Note that the form (\ref{strform}) is indeed non-degenerate even when $M=N$.
\end{remark}

Let $e$ be an even nilpotent element in $\g$. It can be shown (cf. \cite{Wa,Ho}) that there exists (not uniquely) a semisimple element $h\in\g$ such that $\ad h:\g\rightarrow\g$ gives a {\em good $\mathbb{Z}$-grading of $\g$ for $e$}, which means the following conditions are satisfied:
\begin{enumerate}
\item[(1)] $\ad h(e)=2e$,
\item[(2)] $\g=\bigoplus_{j\in\mathbb{Z}} \g(j)$, where $\g(j):=\{x\in\g|\ad h(x)=jx\}$,
\item[(3)] the center of $\g$ is contained in $\g(0)$,
\item[(4)] $\ad e:\g(j)\rightarrow\g(j+2)$ is injective for all $j\leq -1$,
\item[(5)] $\ad e:\g(j)\rightarrow\g(j+2)$ is surjective for all $j\geq -1$.
\end{enumerate}
In this article we only care about the case where the grading is {\em even}; that is, we always assume that $\g(i)=0$ for all $i\notin 2\mathbb{Z}$.

Define the nilpotent subalgebras of $\g$ as follows:
\begin{equation}\label{mpdef}
\mathfrak{p}:=\bigoplus_{j\geq 0}\g(j), \quad \mathfrak{m}:=\bigoplus_{j<0}\g(j).
\end{equation}
Let $\chi\in\g^*$ be defined by $\chi(y):=(y,e)$, for all $y\in\g$. Then the restriction of $\chi$ on $\mathfrak{m}$ gives a one dimensional $U(\mathfrak{m})$-module. Let $I_\chi$ denote the left ideal of $U(\g)$ generated by $\{a-\chi(a)|a\in\mathfrak{m}\}$. By PBW theorem of $U(\g)$, we have $U(\g)=U(\mathfrak{p})\oplus I_\chi$. Let $\pr_\chi:U(\g)\rightarrow U(\mathfrak{p})$ be the natural projection and we can identify $U(\g)/I_\chi \cong U(\mathfrak{p})$. Furthermore we define the $\chi$-twisted action of $\mathfrak{m}$ on $U(\mathfrak{p})$ by  \[a\cdot y := \pr_\chi([a,y]) \text{ for all }a\in\mathfrak{m} \text{ and } y\in U(\mathfrak{p}).\] An element $y\in U(\mathfrak{p})$ is said to be $\mathfrak{m}$-invariant if $a\cdot y=0$ for all $a\in\mathfrak{m}$.

The {\em finite} W{\em -superalgebra} (which we will omit the term ``finite" from now on) is defined to be the subspace of $\mathfrak{m}$-invariants of $U(\mathfrak{p})$ under the $\chi$-twisted action; that is,
\begin{align*}
\mathcal{W}_{e,h}:=U(\mathfrak{p})^\mathfrak{m}=&\{y\in U(\mathfrak{p})| \pr_\chi ([a,y])=0, \forall a\in\mathfrak{m}\}\\
=&\{y\in U(\mathfrak{p})| \big(a-\chi(a)\big)y\in I_\chi, \forall a\in\mathfrak{m}\}.
\end{align*}
At this point, the definition of $W$-superalgebra depends on the nilpotent element $e$ and a semisimple element $h$ which gives a good $\mathbb{Z}$-grading for $e$.

\begin{example}
If we take the nilpotent element $e=0$, then $\chi=0$, $\g=\g(0)=\mathfrak{p}$, $\mathfrak{m}=0$, $\mathcal{W}_{e,h}=U(\mathfrak{p})=U(\g)$.
\end{example}

Now we introduce certain combinatorial objects called {\em $(m,n)$-colored rectangles} \cite{Ho, BBG} (which are in fact special cases of the so called {\em pyramids}). These objects provide a diagrammatic way to record the information needed to define $W$-superalgebras.

Let $\pi$ be a rectangular Young diagram with $m+n$ boxes as its height and $\ell$ boxes as its base. We choose arbitrary $m$ rows and color the boxes in these rows by $+$, while we color the other $n$ rows by $-$. Such a diagram is called an ($m,n$)-{\em colored rectangle}, or a {\em rectangle} for short. For example,
\[
\pi={\begin{picture}(90, 65)%
\put(15,-10){\line(1,0){60}}
\put(15,5){\line(1,0){60}}
\put(15,20){\line(1,0){60}}
\put(15,35){\line(1,0){60}}
\put(15,50){\line(1,0){60}}
\put(15,65){\line(1,0){60}}
\put(15,-10){\line(0,1){75}}
\put(30,-10){\line(0,1){75}}
\put(45,-10){\line(0,1){75}}
\put(60,-10){\line(0,1){75}}
\put(75,-10){\line(0,1){75}}
\put(18,-5){$-$}\put(33,-5){$-$}\put(48,-5){$-$}\put(63,-5){$-$}
\put(18,10){$-$}\put(33,10){$-$}\put(48,10){$-$}\put(63,10){$-$}
\put(18,25){$+$}\put(33,25){$+$}\put(48,25){$+$}\put(63,25){$+$}
\put(18,40){$-$}\put(33,40){$-$}\put(48,40){$-$}\put(63,40){$-$}
\put(18,55){$+$}\put(33,55){$+$}\put(48,55){$+$}\put(63,55){$+$}
\end{picture}}\\[4mm]
\]

Set $M=m\ell$ and $N=n\ell$. We enumerate the \begin{young}(+)\end{young} boxes by the numbers $\{\pa{1},\ldots, \pa{M}\}$ down columns from left to right, and enumerate the \begin{young}(-)\end{young} boxes by the numbers $\{1,\ldots,N\}$ in the same fashion. In fact, we may enumerate the boxes by an arbitrary order as long as the parities are preserved so we just choose the easiest way according to our purpose. Moreover, we image that each box of $\pi$ is of size $2\times 2$ and $\pi$ is built on the $x$-axis where the center of $\pi$ is exactly on the origin. For example,
\begin{equation}\label{exp}
\pi={\begin{picture}(90, 110)%
\put(15,-10){\line(1,0){80}}
\put(15,10){\line(1,0){80}}
\put(15,30){\line(1,0){80}}
\put(15,50){\line(1,0){80}}
\put(15,70){\line(1,0){80}}
\put(15,90){\line(1,0){80}}
\put(15,-10){\line(0,1){100}}
\put(35,-10){\line(0,1){100}}
\put(55,-10){\line(0,1){100}}
\put(75,-10){\line(0,1){100}}
\put(95,-10){\line(0,1){100}}
\put(23,-3){$3$}\put(43,-3){$6$}\put(63,-3){$9$}\put(80,-3){$12$}
\put(23,16){$2$}\put(43,16){$5$}\put(63,16){$8$}\put(80,16){$11$}
\put(23,35){$\pa{2}$}\put(43,35){$\pa{4}$}\put(63,35){$\pa{6}$}\put(82,35){$\pa{8}$}
\put(23,55){$1$}\put(43,55){$4$}\put(63,55){$7$}\put(80,55){$10$}
\put(23,75){$\pa{1}$}\put(43,75){$\pa{3}$}\put(63,75){$\pa{5}$}\put(82,75){$\pa{7}$}
\put(0,-20){\line(1,0){110}}
\put(53,-23){$\bullet$}
\put(63,-35){$1$}\put(83,-35){$3$}
\put(35,-35){$-1$}\put(15,-35){$-3$}
\end{picture}}\\[10mm]
\end{equation}

We now explain how to read off an even nilpotent $e(\pi)$ and a semisimple $h(\pi)$ in $\g=\glMN$ which gives a good $\mathbb{Z}$-grading of $\g$ for $e$ from a given rectangle $\pi$. Let $J=\{\pa{1}<\ldots<\pa{M}<1<\ldots<N\}$ be an ordered index set and let $\{e_i|i\in J\}$ be a basis of $\mathbb{C}^{M|N}$. We identify $\glMN\cong$ End $(\mathbb{C}^{M|N})$ with the $(M+N)\times(M+N)$ matrices over $\mathbb{C}$ by this basis of $\mathbb{C}^{M|N}$ with respect to the order $e_i<e_j$ if $i<j$ in $J$. 

Define the element 
\begin{equation}\label{edef}
e(\pi):=\sum_{\substack{i,j\in J}}e_{i,j}\in \glMN,
\end{equation}
summing over all adjacent pairs $\young(ij)$ of boxes in $\pi.$ It is clear that such an element $e(\pi)$ is even nilpotent.

Let $\widetilde{\text{col}}(i)$ denote the $x$-coordinate of the box numbered with $i\in J$. For instance, in our example (\ref{exp}),  
$\widetilde{\text{col}}(\pa{1})=-3$ and $\widetilde{\text{col}}(8)=1$.
Then we define the following diagonal matrix
\begin{equation}\label{hdef}
h(\pi):=-\text{diag}\big(\widetilde{\text{col}}(\pa{1}), \ldots, \widetilde{\text{col}}(\pa{M}), \widetilde{\text{col}}(1),\ldots, \widetilde{\text{col}}(N)\big)\in\g.
\end{equation}
One may check directly that $\ad h(\pi)$ gives a good $\mathbb{Z}$-grading of $\g$ for $e(\pi)$. 

\begin{remark}\label{choiceofh}
In general, there are other even good $\mathbb{Z}$-gradings for $e$. However, if our $e=e(\pi)$ is obtained by a rectangle $\pi$ according to (\ref{edef}), then such a grading is unique (cf. \cite[Theorem 7.2]{Ho}).
\end{remark}

Now we characterize those $e$ obtained from (\ref{edef}) for some rectangle $\pi$. Consider 
\begin{equation}\label{edecomp}
e=e_M\oplus e_N\in \End \mathbb{C}^{M|N},
\end{equation}
where $e_M$ and $e_N$ are the restrictions of $e$ to $\mathbb{C}^{M|0}$ and $\mathbb{C}^{0|N}$, respectively. Let $\mu=(\mu_1,\mu_2,\ldots)$ and $\nu=(\nu_1,\nu_2,\ldots)$ be the partitions representing the Jordan type of $e_M$ and $e_N$, respectively.

\begin{definition}
An element $e$ is called {\em rectangular} if it is even nilpotent and the Jordan blocks of $e_M$ and $e_N$ are all of the same size $\ell$, i.e., $\mu=(\,\overbrace{\ell,\ldots,\ell}^{m-copies}\,)$ and $\nu=(\,\overbrace{\ell,\ldots,\ell}^{n-copies}\,)$ for some  for some non-negative integers $\ell,m,n$.
\end{definition}

Clearly, $e$ is of the form (\ref{edef}) for some rectangle $\pi$ if and only if $e$ is rectangular. Assume now $e$ is rectangular. We define a new partition $\la$ by collecting all parts of $\mu$ and $\nu$ together and reorder them by an arbitrary order. Since all the parts are the same number $\ell$, we use $\pa{\ell}$ to denote the parts obtained from $\mu$.

For example, consider $\gl_{8|12}$, $\mu=(\pa{4},\pa{4})$ and $\nu=(4,4,4)$. Then one possible $\la$ is $\la=(\pa{4},4,\pa{4},4,4)$.

Next we read off an $\ep$-$\del$ sequence $\mb$ from $\la:$ if the $i$-th position of $\la$ is $\ell$ (respectively, $\pa{\ell}$), then the $i$-th position of $\mb$ is $\ep$ (respectively, $\del$). For example, the $\la$ above corresponds to the $\ep$-$\del$ sequence $\mb=\del\ep\del\ep\ep$. 

Then we color the rectangle of height $m+n$ base $\ell$ with respect to $\mb$: we color the $i$-th row of $\pi$ by + (respectively, $-$) if the $i$-th position of $\mb$ is $\del$ (respectively, $\ep$) where the rows are counted from top to bottom. After coloring the rows, we enumerate the boxes of $\pi$ by exactly the same fashion explained in the paragraph before (\ref{exp}).

Therefore, we have a bijection between the set of $(m,n)$-colored rectangles of base $\ell$ and the set of pairs $(e,\mb)$ where $e$ is a rectangular element in $\gl_{m\ell|n\ell}$ and $\mb$ is an $\ep$-$\del$ sequence containing exactly $m$ $\del$'s and $n$ $\ep$'s.

Let $\pi$ be a fixed $(m,n)$-colored rectangle and $e(\pi)$ denote the rectangular element associates to $\pi$. We will denote by $\mathcal{W}_\pi:=\mathcal{W}_{e(\pi)}$ the $W$-superalgebra associated to $e(\pi)$. Note that we may omit $h(\pi)$ in our notation with no ambiguity due to Remark \ref{choiceofh}.

\begin{remark}\label{Wind}
An interesting observation is that $\W_{\pi}$ is independent of the choices of the sequence $\mb$ because any other sequence $\mb^\prime$ yields to the same $e$ and hence the same $W$-superalgebra. This is parallel to the fact that $Y_{m|n}^\ell$ is independent of the choice of $\mb$.
\end{remark}

Now we label the columns of $\pi$ from left to right by $1,\ldots, \ell$, and for any $i\in J$ we let col($i$) denote the column where $i$ appears. Define the {\em Kazhdan filtration} of $U(\glMN)$
\[
\cdots\subseteq F_dU(\glMN) \subseteq F_{d+1}U(\glMN)\subseteq \cdots
\] 
by declaring 
\begin{equation}\label{degdef}
\deg (e_{i,j}):= \text{col}(j)-\text{col}(i)+1
\end{equation}
for each $i,j \in J$ and $F_dU(\glMN)$ is the span of all supermonomials $e_{i_1,j_1}\cdots e_{i_s,j_s}$ for $s\geq 0$ and $\sum_{k=1}^s$ deg $(e_{i_k,j_k})\leq d$. Let $\gr U(\glMN)$ denote the associated graded superalgebra. The natural projection $\glMN\twoheadrightarrow\mathfrak{p}$ induces a grading on $\W_\pi$. 

On the other hand, let $\g^e$ denote the centralizer of $e$ in $\g=\glMN$ and let $S(\g^e)$ denote the associated supersymmetric algebra. We define the Kazhdan filtration on $S(\g^e)$ by the same setting (\ref{degdef}). The following proposition was observed in \cite{Zh}, where the mild assumption on $e$ there is satisfied when $e$ is rectangular.

\begin{proposition}\label{dimprop}\cite[Remark 3.9]{Zh}
$\gr \W_{\pi}\cong S(\g^e)$ as Kazhdan graded superalgebras.
\end{proposition}

The following proposition is a well-known result about $\g^e$. As remarked in \cite{BBG}, the result is similar to the Lie algebra case $\gl_{M+N}$ since $e$ is even.
\begin{proposition}\label{counting2}
Let $\pi$ be an (m,n)-colored rectangle and $e=e(\pi)$ be the associated rectangular nilpotent element. For all $1\leq i,j\leq m+n$ and $r>0$, define
\[
c_{i,j}^{(r)}:=
\sum_{\substack{1\leq h,k\leq m+n \\ \row(h)=i, \row(k)=j\\ \col(k)-\col(h)=r-1}} (-1)^{|i|}e_{h,k}\in \g=\glMN.
\]
Then the set of vectors $\lbrace c_{i,j}^{(r)}|1\leq i,j\leq m+n, 1\leq r\leq \ell \rbrace$ forms a basis for $\g^e$.
\end{proposition}

\begin{corollary}\label{dimcoro}
Consider $Y_{m|n}^{\ell}$ with the canonical filtration and $S(\g^e)$ with the Kazhdan filtration. Let $F_dY_{m|n}^{\ell}$ and $F_dS(\g^e)$ denote the associated filtered algebras, respectively. Then for each $d\geq 0$, we have $\dim F_dY_{m|n}^{\ell} = \dim F_dS(\g^e)$.
\end{corollary}
\begin{proof}
Follows from Proposition \ref{PBWSY}, Proposition \ref{counting2} and induction on $d$.
\end{proof}

\section{Isomorphism between $Y_{m|n}^{\ell}$ and $\W_\pi$}
Let $\pi$ be a given $(m,n)$-colored rectangle with base $\ell$ and $\mb$ be the $\ep$-$\del$ sequence determined by the colors of rows of $\pi$. We now define some elements in $U(\mathfrak{p})$. It turns out later that they are $\mathfrak{m}$-invariant, i.e., belong to $\W_\pi$.

For each $1\leq r\leq \ell$, set 
\begin{equation}\label{rhodef}
\rho_r := -(\ell-r)(m-n).
\end{equation}
For $1\leq i,j\leq m+n$, define
\begin{equation}\label{etil}
\tilde e_{i,j} := (-1)^{\col(j)-\col(i)} 
(e_{i,j} + \delta_{ij} (-1)^{|i|}\rho_{\col(i)}),
\end{equation}
where $|i|$ is determined by $\mb$ as in (\ref{parity}). 

One may check that \begin{multline}\label{etilrel}
[\tilde e_{i,j}, \tilde e_{h,k}]
=(\tilde{e}_{i,k} - \delta_{ik} (-1)^{|i|} \rho_{\col(i)})\delta_{hj}\\
- (-1)^{(|i|+|j|)(|h|+|k|)}
\delta_{i,k} (\tilde e_{h,j} - \delta_{hj} (-1)^{|j|}\rho_{\col(j)}).
\end{multline}
Also, for any $1\leq i,j\leq m+n$, we have
\begin{equation}\label{chidef}
\chi (\tilde e_{i,j}) = \left\{
\begin{array}{ll}
(-1)^{|i|+1}&\hbox{if $\row(i)=\row(j)$ and $\col(i) = \col(j)+1$;}\\
0&\hbox{otherwise.}
\end{array}\right.
\end{equation}

For $1\leq i,j\leq m+n$,
we let
$t_{ ij;\mb}^{(0)} := \delta_{i,j}$
and for $1\leq r \leq \ell$, define
\begin{equation}\label{tdef}
t_{ij;\mb}^{(r)}
:=
\sum_{s = 1}^r
\sum_{\substack{i_1,\dots,i_s\\j_1,\dots,j_s}}
(-1)^{|i_1|+\cdots+|i_s|}
 \tilde e_{i_1,j_1} \cdots \tilde e_{i_s,j_s}
\end{equation}
where the second sum is over all $i_1,\dots,i_s,j_1,\dots,j_s\in J$
such that
\begin{itemize}
\item[(1)] $\deg(e_{i_1,j_1})+\cdots+\deg(e_{i_s,j_s}) = r$;
\item[(2)] $\col(i_t) \leq \col(j_t)$ for each $t=1,\dots,s$;
\item[(3)] $\col(j_t) < \col(i_{t+1})$ for each $t=1,\dots,s-1$;
\item[(4)] $\row(i_1)=i$, $\row(j_s) = j$;
\item[(5)]
$\row(j_t)=\row(i_{t+1})$ for each $t=1,\dots,s-1$.
\end{itemize}
First note that these elements depends on the choice of $\mb$. Also, the restrictions (1) and (2) imply that $t_{ij;\mb}^{(r)}$ belongs to $\mathrm{F}_r U(\mathfrak p)$ with respect to the Kazhdan grading. Define the following series for all $1\leq i,j\leq m+n$:
\begin{equation}\label{tseries}
t_{ij;\mb}(u) := \sum_{r = 0}^\ell t_{ij;\mb}^{(r)} u^{-r}
\in U(\mathfrak p) [[u^{-1}]].
\end{equation}

Now we prove that these distinguished elements in $U(\mathfrak{p})$ are in fact $\mathfrak{m}$-invariant under the $\chi$-twisted action. Let $T(Mat_\ell)$ be the tensor algebra of the $\ell\times \ell$ matrices space over $\mathbb{C}$ and $\g=\glMN$ where $M=m\ell$, $N=n\ell$.
For all $1\leq i,j\leq m+n$, define a $\mathbb{C}$-linear map $t_{ij;\mb}:T(Mat_\ell)\rightarrow U(\g)$ inductively by
\[
t_{ij;\mb}(1):=\del_{i,j},\qquad\quad t_{ij;\mb}(e_{a,b}):=(-1)^{|i|}e_{i\star a,j\star b},
\] 
\begin{equation}\label{uptij1}
t_{ij;\mb}(x_1\otimes x_2\otimes\ldots \otimes x_r):= \sum_{1\leq i_1,i_2,\ldots,i_{r-1}\leq m+n} t_{ii_1;\mb}(x_1)t_{i_1i_2;\mb}(x_2)\cdots t_{i_{r-1}j;\mb}(x_r),
\end{equation}
for $1\leq a,b\leq \ell$, $r\geq 1$ and $x_1,\ldots,x_r\in Mat_\ell$, where $i\star a$ is defined to be the number in the $(i,a)$-th position of $\pi$, where we label the rows and columns from top to bottom and from left to right. For example, let $\pi$ be the rectangle in (\ref{exp}), then $t_{23;\mb}(e_{2,4})=(-1)^{|2|}e_{2\star 2, 3\star 4}=-e_{4,\pa{8}}$. For an indeterminate $u$, we extend the scalars from $\mathbb{C}$ to $\mathbb{C}[u]$ in the obvious way.

\begin{lemma}
For each $1\leq i,j,h,k\leq m+n$ and $x,y_1,\ldots, y_r\in Mat_\ell$, we have
\begin{align}\notag
[t_{ij;\mb}(x),t_{hk;\mb}(y_1\otimes\cdots\otimes y_r)]&=\\
\notag (-1)^{|i|\,|j|+|i|\,|h|+|j|\,|h|}\big(& \sum_{s=1}^r t_{hj;\mb}(y_1\otimes\cdots\otimes y_{s-1})t_{ik;\mb}(xy_s\otimes\cdots \otimes y_r)\\
&\,-t_{hj;\mb}(y_1\otimes\cdots\otimes y_{s}x)t_{ik;\mb}(y_{s+1}\otimes\cdots\otimes y_{r})
\big),\label{detcomp1}
\end{align}
where the products $xy_s$ and $y_sx$ on the right are the matrix products in $Mat_\ell$.
\end{lemma}

\begin{proof}
By (\ref{uptij1}) and induction on $r$.
\end{proof}

Introducing the following matrix $A(u)$ with entries in the algebra $T(Mat_\ell)[u]$:
$$
A(u) =
\left(
\begin{array}{ccccc}
u+e_{1,1}+\rho_1 & e_{1,2} & e_{1,3} & \cdots & e_{1,\ell}\\
1 & u+e_{2,2}+\rho_2 & e_{2,3} & &e_{2,\ell}\\
0&1&\ddots& &\vdots\\
\vdots&  & 1 & u+e_{\ell-1,\ell-1}+\rho_{\ell-1} & e_{\ell-1,\ell}\\
0&\cdots & 0 & 1 & u+e_{\ell,\ell}+\rho_\ell
\end{array}
\right).\:
$$
A key observation is that for all $1\leq i,j\leq m+n$ and $0\leq r\leq \ell$, the element $t_{ij;\mb}^{(r)}$ of $U(\mathfrak{p})$ defined by (\ref{tdef}) equals to the coefficient of the term $u^{\ell-r}$ in $t_{ij;\mb}\big(\rdet A(u)\big)$, where 
\[
\rdet A:= \sum_{\tau\in S_l}\sgn(\tau)a_{1,\tau(1)}a_{2,\tau(2)}\cdots a_{l,\tau(l)},
\]
for a matrix $A=(a_{i,j})_{1\leq i,j\leq \ell}$. We also let $A_{p,q}(u)$ stand for the submatrix of $A(u)$ consisting only of rows and columns numbered by $p,\ldots,q$.

\begin{proposition}\label{recminv}
For all $1\leq i,j\leq m+n$, $0\leq r\leq \ell$ and a fixed $\mb$, the elements $t_{ij;\mb}^{(r)}$ of $U(\mathfrak{p})$ are $\mathfrak{m}$-invariant under the $\chi$-twisted action.
\end{proposition}

\begin{proof}
Firstly we observe that the statement is trivial when $\ell\leq 1$, hence we may assume $\ell\geq 2$. Note that $\mathfrak{m}$ is generated by elements of the form $t_{ij;\mb}(e_{c+1,c})$, hence it suffices to show that 
\[
\pr_\chi\Big( \big[t_{ij;\mb}(e_{c+1,c}), t_{hk;\mb}\big(\rdet A(u)\big)\big] \Big)=0
\]
for all $1\leq i,j,h,k\leq m+n$ and $1\leq c\leq \ell-1$. In this proof, we omit the fixed $\mb$ in our notation which shall cause no confusion.

By (\ref{detcomp1}), up to an irrelevant sign, we have
\begin{align*}
\Big[t_{ij}&\big(e_{c+1,c}\big), t_{hk}\big(\rdet A(u)\big)\Big]=\\
&t_{hj}\big(\rdet A_{1,c-1}(u)\big) t_{ik}(\rdet \left(
\begin{array}{cccc}
e_{c+1,c} & e_{c+1,c+1} & \cdots & e_{c+1,\ell}\\
1 & u+e_{c+1,c+1}+\rho_{c+1} & \cdots &e_{c+1,\ell}\\
\vdots & &\ddots &\vdots\\
0& \cdots & 1 & u+e_{\ell,\ell}+\rho_{\ell} 
\end{array}
\right))\\
&-t_{hj}(\rdet \left(
\begin{array}{cccc}
u+e_{1,1}+\rho_1 & \cdots & e_{1,c} & e_{1,c}\\
1 & \ddots &   &\vdots\\
\vdots & & u+e_{c,c}+\rho_c & e_{c,c}\\
0& \cdots & 1 & e_{c+1,c} 
\end{array}
\right)) t_{ik}\big(\rdet A_{c+2,\ell}(u)\big).
\end{align*}
A crucial observation is that for $1\leq i,j\leq m+n$ and $1\leq c\leq \ell-1$, we have
\begin{equation}\label{crue}
 t_{ij}\big(e_{c+1,c}(u+e_{c+1,c+1}+\rho_{c+1})\big)-t_{ij}\big(u+e_{c+1,c+1}+\rho_c\big) \equiv 0 \quad (\text{mod} \,\,I_\chi).
\end{equation}
Indeed, it is enough to check (\ref{crue}) when $\ell=2$. The trick here is to rewrite the quadratic terms in $U(\mathfrak{g})$ by supercommtator relations. Then the term $\rho_c$ gives exactly the required constant so that the left-hand side of (\ref{crue}) can be expressed by a linear combination of elements in $I_{\chi}$.

Therefore, 
\begin{align*}
\Big[t_{ij}&\big(e_{c+1,c}\big), t_{hk}\big(\rdet A(u)\big)\Big]\equiv\\
&t_{hj}\big(\rdet A_{1,c-1}(u)\big) t_{ik}(\rdet \left(
\begin{array}{cccc}
1 & e_{c+1,c+1} & \cdots & e_{c+1,\ell}\\
1 & u+e_{c+1,c+1}+\rho_{c} & \cdots &e_{c+1,\ell}\\
\vdots & &\ddots &\vdots\\
0& \cdots & 1 & u+e_{\ell,\ell}+\rho_{\ell} 
\end{array}
\right))\\
&-t_{hj}(\rdet \left(
\begin{array}{cccc}
u+e_{1,1}+\rho_1 & \cdots & e_{1,c} & e_{1,c}\\
1 & \ddots &   &\vdots\\
\vdots & & u+e_{c,c}+\rho_c & e_{c,c}\\
0& \cdots & 1 & 1 
\end{array}
\right)) t_{ik}\big(\rdet A_{c+2,\ell}(u)\big) 
\end{align*}
modulo $I_\chi$. Making the obvious row and column operations gives that 
\[
\rdet \left(
\begin{array}{cccc}
1 & e_{c+1,c+1} & \cdots & e_{c+1,\ell}\\
1 & u+e_{c+1,c+1}+\rho_{c} & \cdots &e_{c+1,\ell}\\
\vdots & &\ddots &\vdots\\
0& \cdots & 1 & u+e_{\ell,\ell}+\rho_{\ell} 
\end{array}
\right)=(u+\rho_c)\rdet A_{c+2,\ell}(u),
\]
\[
\rdet \left(
\begin{array}{cccc}
u+e_{1,1}+\rho_1 & \cdots & e_{1,c} & e_{1,c}\\
1 & \ddots &   &\vdots\\
\vdots & & u+e_{c,c}+\rho_c & e_{c,c}\\
0& \cdots & 1 & 1 
\end{array}
\right)=(u+\rho_c)\rdet A_{1,c-1}(u).
\]
Substituting these back shows that 
$\pr_\chi \Big( \big[t_{ij}(e_{c+1,c}), t_{hk}\big(\rdet A(u)\big)\big] \Big) =0$.
\end{proof}

Set a new notation 
$t_{ij;\mb}(e_{r,r}) = (-1)^{|i|} e_{i,j}^{[r]}\in U(\mathfrak{h})$, where $\mathfrak{h}\cong ({\gl_{m|n}})^{\oplus \ell}$.
Clearly, $\mathfrak{h}$ has a basis $\{ e_{i,j}^{[r]}|1\leq r\leq \ell, 1\leq i,j\leq m+n\}$.
Define \[\eta:U(\mathfrak{h})\rightarrow U(\mathfrak{h}), \quad e_{i,j}^{[r]}\mapsto e_{i,j}^{[r]}-\del_{ij}\rho_r.\]

Let $\xi:U(\mathfrak{p})\rightarrow U(\mathfrak{h})$ be the algebra homomorphism induced from the natural projection $\mathfrak{p}\twoheadrightarrow \mathfrak{h}$. Define the map 
$\mu:= \eta \circ \xi :U(\mathfrak{p})\rightarrow U(\mathfrak{h})\cong U(\mathfrak{gl}_{m|n})^{\otimes \ell}$.
 
\begin{lemma}\label{muimg}
For $1\leq i,j\leq m+n$ and $r>0$,
\[
\mu(t_{ij;\mb}^{(r)})=\sum_{1\leq s_1<\cdots<s_r\leq \ell}\,\,\sum_{1\leq i_1,\ldots, i_{r-1}\leq m+n} 
(-1)^{|i|+|i_1|+\cdots+|i_{r-1}|}
e_{i,i_1}^{[s_1]}e_{i_1,i_2}^{[s_2]}\cdots e_{i_{r-1},j}^{[s_r]}.
\]
\end{lemma}

\begin{proof}
Applying the map $e_{r,s}\mapsto \del_{r,s} (e_{r,r}-\rho_r)$ to the matrix $A(u)$ gives a diagonal matrix with determinant $(u+e_{1,1})(u+e_{2,2})\cdots (u+e_{\ell,\ell})$, where its $u^{\ell-r}$-coefficient equals to
$\sum_{1\leq s_1<\cdots<s_r\leq \ell} e_{s_1,s_1}e_{s_2,s_2}\cdots e_{s_r,s_r}$. Now the lemma follows from (\ref{uptij1}).
\end{proof}

Now we are ready to state and prove our main result of this article.
\begin{theorem}\label{main}
Let $\pi$ be an $(m,n)$-colored rectangle and $\mb$ be the $\ep$-$\del$ sequence determined by the rows of $\pi$. Then there exists an isomorphism $Y_{m|n}^{\ell}\cong \W_\pi$ of filtered superalgebras such that the generators \[\{t_{ij;\mb}^{(r)}|1\leq i,j\leq m+n, 1\leq r\leq \ell\}\] of $Y_{m|n}^{\ell}$ are sent to the elements of $\W_\pi$ with the same names defined by (\ref{tdef}).
\end{theorem}

\begin{proof}
Again, the result is trivial when $\ell\leq 1$ so we assume that $\ell\geq 2$. By Proposition \ref{PBWSY}, the set of all supermonomials in the elements \[\{t_{ij;\mb}^{(r)}|1\leq i,j\leq m+n, 1\leq r\leq \ell\}\] of $Y_{m|n}^{\ell}$ taken in some fixed order and of total degree $\leq d$ forms a basis for $F_d Y_{m|n}^{\ell}$ (with respect to the canonical filtration). Since $\kappa_\ell$ is injective, by Corollary \ref{dimcoro} we have
\[
\dim \kappa_\ell(F_dY_{m|n}^{\ell}) = \dim F_d Y_{m|n}^{\ell} = \dim F_d S(\g^e),
\]
where $S(\g^e)$ is equipped with the Kazhdan grading.

Let $X_d$ denote the subspace of $U(\mathfrak{p})$ spanned by all supermonomials in the elements $\{t_{ij;\mb}^{(r)}|1\leq i,j\leq m+n, 1\leq r\leq \ell\}$ defined by (\ref{tdef}) taken in some fixed order and of total degree $\leq d$. By Lemma \ref{muimg} and the discussion above, we have $\mu(X_d)=\kappa_\ell(F_dY_{m|n}^{\ell})$.

Proposition \ref{recminv} assures that $X_d\subseteq F_d\W_\pi$. Together with Proposition \ref{dimprop}, we have 
\[
\dim F_dS(\g^e)=\dim \mu(X_d)\leq \dim X_d\leq \dim F_d\W_\pi\leq \dim F_dS(\g^e).
\]
Thus equality holds everywhere and hence $X_d=F_d\W_\pi$. Moreover, $\mu$ is injective and $\mu(t_{ij;\mb}^{(r)})=\kappa_\ell(t_{ij;\mb}^{(r)})$ for all $1\leq i,j\leq m+n$ and $0\leq r\leq \ell$. The composition $\mu^{-1}\circ \kappa_\ell:Y_{m|n}^{\ell}\rightarrow \W_\pi$ gives the required isomorphism.
\end{proof}

\begin{remark}
The connection between $W$-algebras of $so(n)$ and $sp(n)$ and their corresponding (twisted) Yangians has not been studied in full generality. Some partial results can be found in \cite{BR2,Ra}; see also \cite{Br} for an approach similar to this article.   
\end{remark}

\subsection*{Acknowledgements}
The author would like to thank the anonymous reviewers
for their valuable comments and suggestions to improve the
quality of this article. The author is supported by 
post-doctoral fellowship of the 
Institution of Mathematics, Academia Sinica, Taipei, Taiwan.


\begin{thebibliography}{2nd}

\bibitem[1]{Br}
J. Brown,
Twisted Yangians and finite $W$-algebras,
{\em Transform. Groups},
{\bf 14} (2009), 87-114

\bibitem[2]{BBG}
J. Brown, J. Brundan and S. Goodwin,
Principal $W$-algebras for GL$(m|n)$,
arXiv:math/1205.0992v2, 
to appear in {\em Algebra and Number Theory}.


\bibitem[3]{BK}
J. Brundan and A. Kleshchev,
Shifted Yangians and finite $W$-algebras,
{\em Advances Math.}
{\bf 200} (2006), 136-195.





\bibitem[4]{BR1}
C. Briot and E. Ragoucy,
RTT presentation of finite $W$-algebras,
{\em J. Phys. A}
{\bf 34} (2001), 7287-7310

\bibitem[5]{BR2}
C. Briot and E. Ragoucy,
Yangians and $W$-algebras,
{\em Theor. Math. Phys.}
{\bf 127} (2001), 709-718



\bibitem[6]{BR3}
C. Briot and E. Ragoucy,
$W$-superalgebras as truncations of super-Yangians,
{\em J. Phys. A}
{\bf 36} (2003), 1057-1081





\bibitem[7]{CW}
S.-J. Cheng and W. Wang,
{\em Dualities and Representations of Lie Superalgebras}, 
Graduate Studies in Mathematics 
{\bf 144}, AMS, 2013.



\bibitem[8]{FSS1}
L. Frappat, A. Sciarrino, P. Sorba,
Structure of basic Lie superalgebras and of their affine extensions, 
{\em Comm. Math. Phys.}
{\bf 121}, (1989), 457-500


\bibitem[9]{FSS2}
L. Frappat, A. Sciarrino, P. Sorba,
{\em Dictionary on Lie algebras and superalgebras}, 
Academic Press (Londres), 2000 



\bibitem[10]{Go}
L. Gow,
Gauss Decomposition of the Yangian $Y(\mathfrak{gl}_{m|n})$,
{\em Comm. Math. Phys.}
{\bf 276} (2007), 799-825.

\bibitem[11]{Ho}
C. Hoyt,
Good gradings of basic Lie superalgebras,
{\em Israel Journal of Mathematics}
{\bf 192} (2012), 251-280.

\bibitem[12]{Ka}
V. Kac,
Lie superalgebras,
{\em Advances Math.}
{\bf 26} (1997), 8-96


\bibitem[13]{LS}
D. Leites, V. Serganova,
Defining relations for classical Lie superalgebra,
{\em Proceedings of the Euler IMI conference on quantum groups},
Leningrad, 1990







\bibitem[14]{Na}
M. Nazarov,
Quantum Berezinian and the classical Capelli identity,
{\em Lett. Math. Phys.} 
{\bf 21} (1991), 123--131.


\bibitem[15]{Pe}
Y. Peng,
Parabolic presentations of the super Yangian $Y(\mathfrak{gl}_{M|N})$,
{\em Comm. Math. Phys.}
{\bf 307}, (2011), 229-259 

\bibitem[16]{Ra}
E. Ragoucy,
Twisted Yangians and folded $W$-algebras,
{\em Internat. J. Modern. Phys.} 
{\bf A16}, (2001), 2411-2433

\bibitem[17]{RS}
E. Ragoucy and P. Sorba,
Yangian realizations from finite $W$-algebras,
{\em Comm. Math. Phys.}
{\bf 203} (1999), 551-572

\bibitem[18]{Wa}
W. Wang,
Nilpotent orbits and finite $W$-algebras,
{\em Fields Institute Communications Series} {\bf 59} (2011),  71--105.





\bibitem[19]{Zh}
L. Zhao,
Finite $W$-superalgebras for queer Lie superalgebras and higher Sergeev duality,
arXiv:math/1012.2326v2

\end{thebibliography}
\end{document}